\documentclass[10pt,fleqn]{amsart}

\usepackage[utf8]{inputenc}
\usepackage[T1]{fontenc}

\usepackage{amsmath,amssymb,latexsym}
\usepackage[mathscr]{eucal}
\usepackage{accents}

\usepackage{cite}

\usepackage{eqnnumwarn}

\theoremstyle{plain}
\newtheorem{theorem}{Theorem}

\numberwithin{equation}{section}

\def\mydate{\number\year-\ifnum\month<10{0}\fi\number\month-\ifnum\day<10{0}\fi\number\day}

\newcommand{\dy}{\partial}
\newcommand{\ddt}[1]{\frac{\mathrm{d}{#1}}{\mathrm{d}{t}}}

\newcommand{\sfrac}[2]{{\textstyle\frac{#1}{#2}}}

\newcommand{\tssum}{{\textstyle\sum}}

\newcommand{\Zahl}{\mathbb{Z}}
\newcommand{\Real}{\mathbb{R}}
\newcommand{\Comp}{\mathbb{C}}

\newcommand{\Oh}{{\sf O}}

\newcommand{\ex}{\mathrm{e}}
\newcommand{\im}{\mathrm{i}}
\newcommand{\eps}{\varepsilon}
\newcommand{\vfi}{\varphi}
\newcommand{\vtt}{\vartheta}
\newcommand{\ups}{\upsilon}

\newcommand{\gb}{\nabla}

\newcommand{\aand}{\quad\textrm{and}\quad}

\newcommand{\sump}[1]{\mathop{\smash{\mathop{{\sum}_{#1}'}}{\vphantom\sum}}}

\newcommand{\Iff}{\>\Leftrightarrow\>}

\newcommand{\hbeta}{{\hat\beta}}

\newcommand{\Dom}{D}

\newcommand{\kpg}{\kappa_g}
\newcommand{\Proj}{{\sf P}}
\newcommand{\Zp}{\mathbb{Z}^3\backslash\{0\}}

\newcommand{\tht}{\theta}
\newcommand{\ds}{\;\mathrm{d}s}

\newcommand{\Unif}{\mathcal{U}}
\newcommand{\CE}{\mathcal{E}}
\newcommand{\CR}{\mathcal{R}}
\newcommand{\mxr}{\Xi}
\newcommand{\ppt}{\varsigma}
\newcommand{\ppz}{\varrho}
\newcommand{\EV}{{\sf E}}
\newcommand{\Var}{{\sf var}}

\newcommand{\rmd}{\mathcal{E}}

\newcommand{\Pkk}{\Proj_{\kappa,2\kappa}}

\newcommand{\dr}{\;\mathrm{d}r}
\newcommand{\drho}{\;\mathrm{d}\rho}
\newcommand{\dvfi}{\;\mathrm{d}\vfi}
\newcommand{\dphi}{\;\mathrm{d}\phi}

\newcommand{\ggfo}[1]{\|{#1}\|_{l_1}^{}}
\newcommand{\ggfosq}[1]{\|{#1}\|_{l_1}^2}

\newcommand{\gs}{{\scriptscriptstyle{\gtrsim}}}

\begin{document}

\title[Batchelor--Howells--Townsend spectrum in 3d]
{The Batchelor--Howells--Townsend spectrum:\\three-dimensional case}
\author{M. S. Jolly$^{1}$}
\address{$^1$Mathematics Department\\
Indiana University\\ Bloomington, IN 47405, United States}
\author{D. Wirosoetisno$^{2}$}
\address{$^2$Mathematical Sciences\\
Durham University\\ Durham, United Kingdom\ \ DH1 3LE}
\email[M. S. Jolly]{msjolly@indiana.edu}
\email[D. Wirosoetisno]{djoko.wirosoetisno@durham.ac.uk}

\thanks{This work was supported in part by NSF grant number DMS-1818754, and the Leverhulme Trust grant VP1-2015-036. }


\begin{abstract}
Given a velocity field $u(x,t)$, we consider the evolution of a passive tracer $\tht$ governed by $\dy_t\tht + u\cdot\gb\tht = \Delta\tht + g$ with time-independent source $g(x)$.
When $\|u\|$ is small in some sense, Batchelor, Howells and Townsend (1959, J.\ Fluid Mech.\ 5:134; henceforth BHT) predicted that the tracer spectrum scales as $|\tht_k|^2\propto|k|^{-4}|u_k|^2$.
Following our recent work for the two-dimensional case, in this paper we prove that the BHT scaling does hold probabilistically, asymptotically for large wavenumbers and for small enough random synthetic three-dimensional incompressible velocity fields $u(x,t)$.
We also relaxed some assumptions on the velocity and tracer source, allowing finite variances for both and full power spectrum for the latter.
\end{abstract}

\maketitle

\section{Introduction}\label{s:intro}

Several theories relate the spectrum of a passive tracer to the energy spectrum of the velocity that advects it.
When inertial effects dominate (tracer) diffusion, which happens at larger scales for sufficiently large velocity, the Batchelor--Corrsin--Obukhov theory \cite{batchelor:59,corrsin:51,obukhov:49} has the tracer spectrum scale as $|k|^{-(\hbeta+5)/2}$, for a (``turbulent'') energy spectrum scaling as $\CE(k)\sim |k|^\hbeta$.
On the other hand, tracer diffusion will inevitably dominate at smaller scales, where it was predicted by Batchelor, Howells and Townsend \cite{batchelor-howells-townsend:59}, and later, using a different argument, by Kraichnan \cite{kraichnan:68,kimura-kraichnan:93} that the tracer spectrum should scale as $|k|^{-4}\CE(k)$.
The authors in \cite{batchelor-howells-townsend:59} assumed that over this range, the characteristic time of velocity is much greater than that of the tracer.
With the source term in the equation for the tracer being approximately steady, they drop the time derivative.
Kraichnan's argument has the velocity field fluctuating randomly in time with a correlation time much shorter than both the diffusive and convective time scales.
For synopses and later developments, see, e.g., \cite{vallis:aofd,shraiman-siggia:00,warhaft:00,gotoh-watanabe-suzuki:11,sreenivasan:19}.
There is also considerable computational evidence for the BHT spectrum \cite{chasnov-canuto-rogallo:88,chasnov:91,holzer-siggia:94,gotoh-watanabe-suzuki:11,yeung-sreenivasan:13,yeung-sreenivasan:14}.

Following our earlier paper on the two-dimensional case \cite{jolly-dw:bht}, in this paper we make rigorous the intuitive arguments of Batchelor--Howells--Townsend and Kraichnan, proving that the BHT scaling does hold probabilistically, asymptotically for large wavenumbers, for a class of small random synthetic three-dimensional incompressible velocity fields.
As in \cite{jolly-dw:bht}, we also confirm the intuition of \cite{batchelor-howells-townsend:59,kraichnan:68} that this holds whether or not the velocity is time dependent, subject to reasonable assumptions.
Relaxing some assumptions in \cite{jolly-dw:bht}, here we allow a more general modal random variable, only requiring it to be circular, and a more general tracer (variance) source, only requiring sufficient Sobolev regularity.
The main limitation of our present result, specifically the convergence proof of the (unknown) remainder, is that it requires an energy spectrum no shallower than $|k|^{-2}$, which rules out an application to the hypothesized Kolmogorov $|k|^{-5/3}$ energy spectrum.

Our general approach is similar to \cite{jolly-dw:bht}, with some differences:
As in the two-dimensional case treated in \cite{jolly-dw:bht}, our approach is to calculate precisely the expectation of the spectrum for the first iterate of a fixed point iteration of the tracer, and show that the error from the actual tracer can be made small by taking the velocity small.  This, coupled with a certain bound on the variance, establish the BHT scaling.
Unlike the 2d case, however, here we consider a random tracer source at all scales, except for the bound on the variance.
Also different is that rather than randomizing the phases of the streamfunction in Fourier space, here we randomize components of the coefficients in the Craya--Herring 3d basis \cite{craya:58,herring:74,kimura-herring:12}.

Through its obvious relationship to various Sobolev norms, the tracer spectrum is related to the degree of tracer mixing and of the efficiency of mixing by the advecting velocity \cite{thiffeault-doering-gibbon:04,doering-thiffeault:06,thiffeault-doering:11,miles-doering:18j,miles-doering:18n,oakley-thiffeault-doering:21}.
A steep tracer spectrum such as the BHT suggests poor mixing, either throughout the entire range for small velocity (treated here), or beyond the so-called diffusive wavenumber \cite[(8.108)]{vallis:aofd} for larger velocity (with the Obukhov--Corrsin/Batchelor tracer spectrum at larger scales).

\section{Preliminaries}\label{s:prelim}

We consider the evolution of a passive scalar $\tht(x,t)$ under
a prescribed velocity field $u(x,t)$ and source $g(x)$,
\begin{equation}\label{q:dtht}
   \dy_t\tht + u\cdot\gb\tht = \Delta\tht + g.
\end{equation}
For simplicity, we take $x\in\Dom:=[0,2\pi]^3$ and assume periodic boundary conditions in all directions.
With no loss of generality, we assume that, for all $t$
\begin{equation}\label{q:mean0}
   \int_\Dom u(x,t) \;\mathrm{d}x = 0
   \quad\text{and}\quad
   \int_\Dom \tht(x,t) \;\mathrm{d}x = 0.
\end{equation}
We note that for the latter to hold for all $t>0$, we must impose the same condition on $g$ and $\tht(\cdot,0)$.

We expand $\tht(x,t)$ in Fourier series as
\begin{equation}
  \tht(x,t) = \tssum_k'\,\tht_k(t)\ex^{\im k\cdot x},
\end{equation}
where the prime indicates that the sum is taken over $k\in\Zp$ to satisfy \eqref{q:mean0}, and denote spectral projection by
\begin{equation}\label{q:Pdef}
   (\Proj_{\kappa,\kappa'}\tht)(x,t) := \sum\nolimits_{\kappa\le|k|<\kappa'} \tht_k(t)\ex^{\im k\cdot x}.
\end{equation}

For the tracer (variance) source, we take either the deterministic
\begin{equation}\label{q:gdef}
  (\Delta^{-1}g)(x) = \tssum_k'\,\gamma_k\ex^{\im k\cdot x}
\end{equation}
where $\gamma_k\in\Comp$ with $\gamma_{-k}=\overline{\gamma_k}$, $\gamma_0=0$ and
\begin{equation}\label{q:gamma}
  |\gamma_k| \le c_g|k|^\alpha
  \qquad\text{when }|k|\ge\kpg
\end{equation}
for some constants $c_g\ge0$, $\kpg>1$ and $\alpha<0$.
The case $c_g=0$ gives the bandwidth-limited source considered in \cite{jolly-dw:bht}, for which somewhat tighter estimates can be obtained below.
Alternately, and without altering the conclusion (see below), one may also consider the random
\begin{equation}\label{q:grdef}
  (\Delta^{-1}g_r)(x) = \tssum_k'\,\gamma_k Z_k\ex^{\im k\cdot x}.
\end{equation}

The complex random variable $Z_k$ is constructed as follows.
For a fixed $k\in\Zahl^3$, we write $Z_k=R_k\ex^{\im\zeta_k}$ where the random phase $\zeta_k\sim\Unif(0,2\pi)$, implying that $Z_k$ is \emph{circular\/}, i.e.\ $\EV(\ex^{\im\phi}Z_k)=\EV Z_k$ for any deterministic real $\phi$.
This in turn implies that $\EV Z_k^n=0$ for any integer $n\ne0$.
We constrain the random modulus $R_k$ so that
\begin{equation}\label{q:r0}
  \mxr := \sup\{s: P(R_k>s) > 0 \} < \infty,
\end{equation}
implying that $Z_k$ is bounded, $|Z_k|\le\mxr$.
(With circularity, this means that $Z_k$ is a {\em proper random variable\/}.)
With no loss of generality, we put as its variance and fourth moment
\begin{equation}\label{q:r24}
  \EV|Z_k|^2 = 1
  \quad\text{and}\quad
  \EV|Z_k|^4 = \ppz.
\end{equation}
We note that by the Cauchy--Schwarz inequality, $1=(\EV|Z_k|^2)^2\le\EV|Z_k|^4=\ppz$, with equality (i.e.\ $\ppz=1$) attained iff $|Z_k|=1$ a.s.
We denote $Z_k\sim\CR_\ppz$.

Now for $g(x,t)$ to be real-valued, we must require that $Z_{-k}=\overline{Z_k}$, but otherwise $Z_k$ are assumed to be uncorrelated, so $\EV Z_j\overline{Z_k}=\delta_{jk}$ and $\EV Z_j Z_k=\delta_{j,-k}$.
A convenient tool to handle this reality constraint is the wavenumber half-space
\begin{equation}\label{q:z3+}
  \Zahl_+^3:=\{(l,m,n):n>0\}\cup\{(l,m,0):m>0\}\cup\{(l,0,0):l>0\}
\end{equation}
with $l$, $m$, $n\in\Zahl$; we thus have $\Zahl_+^3\cup(-\Zahl_+^3)=\Zp$.
With this, we can write $\EV Z_jZ_k=0$ and $\EV Z_j\overline{Z_k}=\delta_{jk}$ for all $j$, $k\in\Zahl_+^3$.

To set up our velocity, we recall the Craya--Herring basis \cite{craya:58,herring:74,kimura-herring:12}:
Writing a wave\-vector $k=(k_x,k_y,k_z)$, defining $k_h:=(k_x,k_y,0)$ and using spherical coordinates $(1,\vfi,\phi)$, we define the ($k$-dependent) orthonormal vectors
\begin{align*}
    d_k &= \frac{k}{|k|}
	& &= \frac{(k_x,k_y,k_z)}{|k|}
	& &=(\sin\vfi\cos\phi,\sin\vfi\sin\phi,\cos\vfi),\\
    e_k &= \frac{k\times\hat z}{|k\times\hat z|}
	& &= \frac{(k_y,-k_x,0)}{|k_h|}
	& &= (\sin\phi,-\cos\phi,0),\\
    f_k &= \frac{k\times k\times\hat z}{|k\times k\times\hat z|}
	& &= \frac{(k_xk_z,k_yk_z,-|k_h|^2)}{|k|\,|k_h|}
	& &= (\cos\vfi\cos\phi,\cos\vfi\sin\phi,-\sin\vfi)\;,
\end{align*}
where $\hat z= (0,0,1)$.
With these basis vectors, any velocity field $v(x,t)$ can be written as
\begin{equation}
   v(x,t) = \tssum_k\,[\tilde U_d(k,t) d_k + \tilde U_e(k,t) e_k + \tilde U_f(k,t) f_k]\ex^{\im k\cdot x}
\end{equation}
for some  $(\tilde U_d,\tilde U_e,\tilde U_f)\in \Comp^3$.
Now since $\text{div}\,v(x,t)=\im\sum_k\,\tilde U_d(k,t)|k|\ex^{\im k\cdot x}$, for $v(\cdot,t)$ to be incompressible we must have $\tilde U_d(k,t)\equiv0$.
We thus write our incompressible velocity field as
\begin{equation}\label{q:udef}
  u(x,t) = \sump{k} |k|^\beta [U_e e_k V_k(t) + U_f f_k W_k(t)]\,\ex^{\im k\cdot x}
\end{equation}
where $\beta<0$, $U_e$ and $U_f$ are real constants, and $V_k(t)$ and $W_k(t)$ complex random processes whose time behaviour will be precised below.\footnote{We note that in this paper $|u_k|\sim|k|^\beta$ whereas in \cite{jolly-dw:bht} $|u_k|\sim|k|^{\beta+1}$, so the $\beta$s are different; in hindsight, we feel the present notation to be more natural.}
For now, we require that, for each fixed $t$, $V_k(t)$ and $W_k(t)\sim\CR_\ppt$, proper random variables with unit variance and $\EV|V_k(t)|^4=\EV|W_k(t)|^4=\ppt$, and bounded as $|V_k(t)|$, $|W_k(t)|\le\mxr$.
As with $g(x)$, for $u(x,t)$ to be real-valued, we must require that $V_{-k}(t)=\overline{V_k(t)}$ and $W_{-k}(t)=\overline{W_k(t)}$.
Aside from this constraint, we assume that $V_k(t)$ and $W_k(t)$ are uncorrelated, so $\EV V_j(t)\overline{V_k(t)}=\EV W_j(t)\overline{W_k(t)}=\delta_{jk}$ and $\EV V_j(t)\overline{W_k(t)}=0$ for all $j$, $k\in\Zahl_+^3$.
Unlike $d_k$, which gives the divergent component of $u$, the $e_k$ and $f_k$ components have no special meaning when $u$ is isotropic, although they do carry physical significance in, e.g., stratified flows (as the ``vortex'' and ``wave'' components, respectively).

We turn to the energy spectrum.
First, we compute [suppressing dependence on $t$ where no confusion may arise]
\begin{equation}\begin{aligned}
  &\|\Pkk u(\cdot,t)\|_{L^2}^2\\
	&\quad= \bigl(\tssum_j\,|j|^\beta (U_e e_j V_j + U_f f_j W_j)\ex^{\im j\cdot x},\tssum_k\,|k|^\beta(U_e e_k V_k + U_f f_k W_k)\ex^{\im k\cdot x}\bigr)_{L^2}^{}\\
	&\quad= 8\pi^3 \tssum_k\,|k|^{2\beta}( U_e^2|V_k|^2 + U_f^2|W_k|^2),
\end{aligned}\end{equation}
where for the second equality we have used the facts that $(\ex^{\im j\cdot x},\ex^{\im k\cdot x})_{L^2}^{}=(2\pi)^3\delta_{jk}$, $e_k\cdot e_k = f_k\cdot f_k=1$ and $e_k\cdot f_k=0$.
Unlike in our previous work on the 2d case \cite{jolly-dw:bht}, here $\|\Pkk u\|^2$ contains the random variables $|V_k|^2$ and $|W_k|^2$, so we compute
\begin{equation}\label{Pucalc}\begin{aligned}
  \EV\|\Pkk u(\cdot,t)\|_{L^2}^2
	&= 8\pi^3 \tssum_k\,|k|^{2\beta}(U_e^2\EV|V_k|^2 + U_f^2\EV|W_k|^2)\\
	&= 8\pi^3(U_e^2+U_f^2)\, \tssum_{\kappa\le|k|<2\kappa} |k|^{2\beta}.
\end{aligned}\end{equation}
{Approximating the sum as an integral over the corresponding region in $\Real^d$, we find t}his scales as $\kappa^{2\beta+d}$ for sufficiently large $\kappa$, so for the classical Kolmogorov $-\frac53$ spectrum in $d=3$ (i.e.\ $2\beta+d-1=-\frac53$), we must take $\beta=-\frac{11}{6}$.
Next, we compute the variance $\Var\|\Pkk u\|_{L^2}^2$, by first using  \eqref{q:r24} to obtain
\begin{equation*}\begin{aligned}
    &(2\pi)^{-6}\EV\|\Pkk u\|_{L^2}^4
    = \EV\bigl(\tssum_k\,|k|^{2\beta}(U_e^2|V_k|^2 + U_f^2|W_k|^2)\bigr)^2\\
    &\quad= \tssum_{jk}\,|j|^{2\beta}|k|^{2\beta}(U_e^4\EV|V_j|^2|V_k|^2 + 2U_e^2U_f^2\EV|V_j|^2|W_k|^2 + U_f^4\EV|W_j|^2|W_k|^2)\\
    &\quad= \tssum_{j\ne k}\,|j|^{2\beta}|k|^{2\beta}(U_e^4\EV|V_j|^2|V_k|^2 + U_f^4\EV|W_j|^2|W_k|^2)\\
    &\qquad+ \tssum_{k}\,|k|^{4\beta}(U_e^4\EV|V_k|^4 + U_f^4\EV|W_k|^4)
    + 2\,\tssum_{jk}\,|j|^{2\beta}|k|^{2\beta}U_e^2U_f^2\EV|V_j|^2|W_k|^2\\
    &\quad= \tssum_{j\ne k}\,|j|^{2\beta}|k|^{2\beta}(U_e^4 + U_f^4)
    + \ppt\,\tssum_{j=k}\,|j|^{2\beta}|k|^{2\beta}(U_e^4 + U_f^4)\\
    &\qquad+ 2\,\tssum_{jk}\,|j|^{2\beta}|k|^{2\beta}U_e^2U_f^2\\
    &\quad= \bigl((U_e^2+U_f^2)\tssum_k\,|k|^{2\beta}\bigr)^2 + (\ppt-1)(U_e^4+U_f^4)\tssum_k\,|k|^{4\beta},
\end{aligned}\end{equation*}
whence
\begin{equation}\begin{aligned}
  \Var\|\Pkk u\|_{L^2}^2 &= \EV\|\Pkk u\|_{L^2}^4 - \bigl(\EV\|\Pkk u\|_{L^2}^2\bigr)^2\\
  &= (2\pi)^6(\ppt-1)(U_e^4+U_f^4)\tssum_{\kappa\le|k|<2\kappa}\,|k|^{4\beta}.
\end{aligned}\end{equation}
For large $\kappa$, this scales as $\kappa^{4\beta+d}$, so $(\Var\|\Pkk u\|^2)^{1/2}/\EV\|\Pkk u\|^2\propto\kappa^{-d/2}$, giving asymptotic convergence (over dyads) to an energy spectrum that is $\kappa^{2\beta+d-1}$.

For the time dependence, we assume that,
for all $j$, $k\in\Zahl_+^3$,
\begin{equation}\label{q:tcov}\begin{aligned}
    \EV V_j(s)\overline{V_k(t)}
    = \EV W_j(s)\overline{W_k(t)}
    = \delta_{jk}\Phi_k(s-t).
\end{aligned}\end{equation}
We take a time correlation function of the form
\begin{equation}\label{q:Phi}
   \Phi_k(t) = \Phi(\chi_k^{}|t|)
\end{equation}
with $\Phi\in C^n(\Real_+)$ for some $n\ge2$ and $\Phi(0)=1$,
where the correlation timescale $\chi_k^{-1}$ is assumed not to grow too rapidly with $|k|$,
\begin{equation}\label{q:chik}
   \lim_{|k|\to\infty}\chi_k^{}|k|^{-2} = 0.
\end{equation}
Using the Cauchy--Schwarz inequality, we have
\begin{equation}\label{q:bdx2}\begin{aligned}
   |\Phi(h)| &= |\EV\,V_k(s)\overline{V_k(s+h)}|\\
	&\le (\EV\,|V_k(s)|^2)^{1/2}(\EV\,|V_k(s+h)|^2)^{1/2}
	= \Phi(0) = 1.
\end{aligned}\end{equation}
We also assume that $V_k(t)$ has sufficient smoothness in $t$ for the usual Riemann integral to be defined.
As before, $V_j(s)$ and $W_k(t)$ are uncorrelated proper random variables for any $j$, $k\in\Zahl_+^3$, $s$ and $t$.

\section{Main Result and Discussion}\label{s:main}

As in the 2d case, it is both convenient and instructive to first consider the static case
\begin{equation}\label{q:static}
  u\cdot\gb\tht = \Delta\tht + g.
\end{equation}
Here the time-independent random velocity is
\begin{equation}\label{q:usdef}
  u(x) = \tssum_k'\, |k|^\beta [U_e e_k V_k + U_f f_k W_k]\,\ex^{\im k\cdot x}
\end{equation}
with $V_k$ and $W_k\sim\CR_\ppt$ i.i.d.\ satisfying the usual reality constraints.

{The strategy, as in \cite{jolly-dw:bht}, is to solve \eqref{q:static} by a fixed-point iteration, show that the error from the first iterate $\vartheta$ to the solution is, at most, of the same order, as $\kappa \to \infty$ as $\EV \|\Pkk\vtt\|_{L^2}^2$.  The latter will satisfy the BHT spectrum and the {\it relative} error can be made arbitrarily small by taking $U/U_{\text{max}}$ small enough.}
We put $\tht^{(0)}=-\Delta^{-1}g$ and
\begin{equation}\label{q:sfpi}
  \tht^{(n+1)} = \Delta^{-1}(u\cdot\gb\tht^{(n)}-g).
\end{equation}
We seek to prove that this iteration converges under some assumptions, and that the limit $\tht^{(\infty)}$ asymptotes, dyad-wise as $\kappa\to\infty$, probabilistically to the BHT spectrum.
Unlike in \cite{jolly-dw:bht}, however, {here} our source $g$ may have a full spectrum, so $\vtt:=\tht^{(1)}-\tht^{(0)}=-\Delta^{-1}(u\cdot\gb\Delta^{-1}g)$ has a remainder arising from high-frequency parts of $g$.

Denoting $\ggfo{f}:=\tssum_k\,|f_k|$ and putting $U_e=U_f=U$, we have the following:

\begin{theorem}\label{t:static}
Let $u$ be given by \eqref{q:usdef} with $\beta<-2$ and {satisfying} \eqref{q:umax}, $g$ and $g_r$ by \eqref{q:gdef}--\eqref{q:gamma} and \eqref{q:grdef}--\eqref{q:gamma} with $\alpha<2\min\{\beta,-d\}-1$ and $\kpg\ge16$.
Then for $\kappa\ge4\kpg^2$ the static problem \eqref{q:static} has a unique solution $\tht=-\Delta^{-1}g + \vtt + \delta\theta$ where
\begin{equation}
   \EV\|\Pkk\vtt\|_{L^2}^2 = \kappa^{2\beta-1} \frac{8\pi U^2}3\,\frac{2^{2\beta-1}-1}{2\beta-1}\,\|\Proj_{1,\kappa^{1/2}}\gb^{-1}g\|_{L^2}^2 + \rmd(\kappa) \label{q:epkk}\\
\end{equation}
and the remainder terms are bounded as, with $\eps=U/U_{max}$,
\begin{align}
   &|\rmd(\kappa)| \le c_g^2 U^2 c(\alpha,\beta)\,\kappa^\alpha + c(\beta)U^2\|\gb^{-1}g\|_{L^2}^2\kappa^{2\beta-3/2},\label{q:bdE}\\
   &\|\Pkk\delta\tht\|_{L^2}^2 \le \eps^2c(g,\alpha,\beta) U^2\mxr^2\,\kappa^{2\beta-1}.\label{q:bddtht}
\end{align}
With finite-mode source, $c_g=0$ in \eqref{q:gamma}, the variance is bounded from above as
\intomargin
\begin{equation}\label{q:ubV}
   \Var\|\Pkk\vtt\|_{L^2}^2 \lesssim \kappa^{4\beta-5}\,{16\pi U^4}\,\frac{2^{4\beta-5}-1}{4\beta-5} \|\gb^{-1}g\|_{L^2}^2 \bigl\{ \ggfosq{\gb^{-1}g} + (\ppt-1) \|\gb^{-1}g\|_{L^2}^2 \bigr\}.
\end{equation}
\end{theorem}

\noindent {As noted after (1.14), $\EV \|P_{\kappa,2\kappa} u\|_{L^2}^2$ scales as $\kappa^{2\beta +3}$ so that
\begin{equation*}
   \EV \|\Pkk\vtt\|_{L^2}^2/\EV \|\Pkk u\|_{L^2}^2\propto|k|^{-4}\;.
\end{equation*}
As in \cite{jolly-dw:bht}, by $f_1(\kappa)\simeq f_2(\kappa)$ we mean that $\lim_{\kappa\to\infty} f_1(\kappa)/f_2(\kappa)=1$.
Thus, ``$\simeq$'' arises either from lattice effect, when we approximate sums over subsets of $\Zahl^d$ by the corresponding integrals over subsets of $\Real^d$, or from dropping terms
of (relative) order $\kpg/\kappa$.
The same convention will be used for ``$\lesssim$''.
As a consequence, absolute constants are included in such relations.}
We note that with finite-mode sources, $\Proj_{1,\kappa^{1/2}}\gb^{-1}g=\gb^{-1}g$ in \eqref{q:epkk}, while in \eqref{q:bdE} the first term vanishes and the second term can be improved to $\Oh(\kappa^{2\beta-2})$.
These results are stated for the isotropic case, $U_e=U_f$, but we have kept $U_e$ and $U_f$ when computing (the main part of) $\EV\|\Pkk\vtt\|^2$ for readers interested in the effect of non-isotropic velocity.
We see no conceptual difficulty to extend \eqref{q:ubV} to sources with full spectra following the approach for $\EV\|\Pkk\vtt\|^2$, but did not attempt this in order to keep the proof readable.

For the time-dependent case, we write the solution $\tht(x,t)$ of \eqref{q:dtht} as the limit of iterates $\tht^{(n)}(x,t)$ defined by
\begin{align}
   &\tht^{(0)} = -\Delta^{-1}g,\label{q:it0}\\
   &\tht^{(n+1)}(\cdot,t) = -\Delta^{-1}g - \int_0^t \ex^{(t-s)\Delta}[u(\cdot,s)\cdot\gb\tht^{(n)}(\cdot,s)] \ds.\label{q:itn}
\end{align}
Here $\ex^{-t\Delta}$ is the heat kernel, i.e.\ $\tht^{(n+1)}$ is the solution of
\begin{equation}
   (\dy_t - \Delta)\tht^{(n+1)} = g - u\cdot\gb\tht^{(n)}
   \quad\text{with}\quad \tht^{(n+1)}(\cdot,0) = -\Delta^{-1}g.
\end{equation}
Our main result is that this iteration converges, and that the limit obeys the BHT scaling in the following sense:

\begin{theorem}\label{t:main}
Let the source $g(x)$ be given by \eqref{q:gdef}--\eqref{q:gamma} or \eqref{q:grdef}--\eqref{q:gamma} with $\alpha<2\min\{\beta,-d\}-1$, the incompressible velocity $u(x,t)$ by \eqref{q:udef} and \eqref{q:tcov} with $\beta<-2$ and $U$ small enough that the convergence condition \eqref{q:idr} holds.
Then the solution $\tht(x,t)$ of \eqref{q:dtht} can be written as $-\Delta^{-1}g+\vtt+\delta\tht$ where $\vtt(x,t)$ satisfies
\begin{equation}\begin{aligned}
    \lim_{t\to\infty}\EV|\vtt_k(t)|^2 &= |k|^{-4}\tssum_j' \bigl[U_e^2(e_{k-j}\cdot j)^2 + U_f^2(f_{k-j}\cdot j)^2\bigr]\,|k-j|^{2\beta}|\gamma_j|^2 \times{}\\
	&\quad\Bigl[1 + \frac{\chi_{k-j}^{}}{|k|^2}\Phi'(0) + \cdots + \frac{\chi_{k-j}^{{n-1}}}{|k|^{2n}}\int_0^\infty \ex^{-s|k|^2/\chi_{k-j}^{}}\Phi^{(n)}(s)\ds\Bigr].
\end{aligned}\end{equation}
When $\sup_k\{\chi_k^{}\}/\kappa^2\ll1$, this reduces to the static case in Theorem~\ref{t:static}, up to further lower-order remainders.
\end{theorem}
\section{Proofs}\label{s:pf}

\begin{proof}[Proof of Theorem~\ref{t:static}]
This consists of three main parts.
In the first part, we compute $\vtt$ and show that it satisfies \eqref{q:epkk} and \eqref{q:bdE}.
We then bound $\Var\|\Pkk\vtt\|^2$.
In the final part, we estimate $\tht^{(\infty)}-\tht^{(1)}$ to obtain \eqref{q:bddtht}.

\subsection{Computing $\vtt$}

We start {with the computation of} $\vtt=\theta^{(1)}-\theta^{(0)}=-\Delta^{-1}(u\cdot\gb\Delta^{-1}g)$, {which, when combined with the scaling in \eqref{Pucalc}, shows} that it satisfies the BHT scaling up to small remainders.
We use $g_r$ in \eqref{q:grdef}, as will be apparent shortly, with no loss of generality.
In some expressions (notably as exponents), we write $d=3$ and $\omega_3=4\pi$, to give a hint of how the analogues would appear in two dimensions.
From \eqref{q:grdef} and \eqref{q:udef}, we have
\begin{equation}\label{q:vttk}
  \vtt_k = \im\,|k|^{-2}\sum\nolimits_{j}'|k-j|^\beta\gamma_{j}[U_e(e_{k-j}\cdot j) V_{k-j} + U_f(f_{k-j}\cdot j) W_{k-j}] Z_{j}\,.
\end{equation}
In computing $\EV\vtt_k\overline{\vtt_k}$, we find factors of $\EV V_{k-j}\overline{V_{k-i}}$, which is
nonzero {if and only if} $k-j=k-i$, i.e.\ $j=i$.
An analogous reasoning applies to $\EV W_{k-j}\overline{W_{k\pm i}}$, so we have $\EV V_{k-i}\overline{V_{k-j}} = \EV W_{k-i}\overline{W_{k-j}}=\delta_{ij}$.
Recalling that $\EV V_j\overline{W_k}=0\;\forall j,k$ and, by independence of $V_j$ and $Z_k$, $\EV|V_j|^2|Z_k|^2=\EV|V_j|^2\EV|Z_k|^2$, we arrive at
\begin{equation}\label{q:evttk}
  \EV\vtt_k\overline{\vtt_k} = |k|^{-4}\tssum_j^{'}|k-j|^{2\beta}|\gamma_j|^2(U_e^2\xi_{kj}^2 + U_f^2\upsilon_{kj}^2)\EV|Z_j|^2 =: |k|^{-4}S_k
\end{equation}
where $\xi_{kj}:=e_{k-j}\cdot j$ and $\upsilon_{ki}:=f_{k-i}\cdot i$.
Since $\EV|Z_j|^2=1$, it is clear that this expression applies to both deterministic $g$ and random $g_r$.

We fix some $r\in(0,1)$; for concreteness, we put $r=\sfrac12$ in Theorem~\ref{t:static},
but write $r$ in this proof to indicate possible optimisation.
Consider any wavenumber dyad $[\kappa,2\kappa)$ with $\kappa>(2\kpg)^{1/r}$.
For any $k$ within this dyad, $\kappa\le|k|<2\kappa$, we split the sum in \eqref{q:evttk} into (here and below $S_k$ denotes a ``temporary variable'' with no global significance),
\begin{equation}\label{q:sk0}
   S_k = \tssum_{1\le|j|<\kappa^r} + \tssum_{\kappa^r\le|j|}
	=: S_k^{\ll} + S_k^{\gs}.
\end{equation}
We start with the last sum $S_k^\gs$, where by \eqref{q:gamma} and $e_{k-j}$ and $f_{k-j}$ being unit vectors,
\begin{equation}\label{q:auxb0}
  |\gamma_j|^2\le c_g^2|j|^{2\alpha},\qquad
  (e_{k-j}\cdot j)^2\le|j|^2
  \aand
  (f_{k-j}\cdot j)^2\le|j|^2.
\end{equation}
From $|j|\ge\kappa^r$, we have $|j|^{2\alpha+2}\le\kappa^{(2\alpha+2)r}$, so writing $m:=k-j$, we then replace the sum over $|j|\ge\kappa^r$ with one over $m\in\Zp$, {giving}
\intomargin
\begin{equation}
   \sum_{\kappa^r\le|j|}\,c_g^2U^2|j|^{2\alpha+2}|k-j|^{2\beta}
   \le c_g^2U^2\kappa^{(2\alpha+2) r}\sum_m{}^{'}\,|m|^{2\beta}
   \le c_g^2U^2\kappa^{(2\alpha+2) r}\frac{\omega_d}{|2\beta+d|}
\end{equation}
since $2\beta+d<0$.
{In the case of} bandwidth-limited source, $c_g=0$ in \eqref{q:gamma}, so this remainder term is zero.
{For $g$ with full spectrum, we sum} over our dyad {to obtain}
\begin{equation}
   \tssum_{\kappa\le|k|<2\kappa} |k|^{-4}S_k^\gs
	\le c_g^2U^2c(\alpha,\beta,d)\kappa^{(2\alpha+2)r-4+d}.
\end{equation}
For this to be dominated by $\kappa^{2\beta-1}$, we need $(2\alpha+2)r-4+d<2\beta-1$.
Putting $d=3$ and $r=\sfrac12$
gives the first term in \eqref{q:bdE}. 

The first sum in \eqref{q:sk0} is more delicate, requiring tight upper and lower bounds.
We start with a couple of preliminary estimates.
Writing $m:=k-j$ again, we bound
\begin{align}
   &|j| < \kappa^r \le |k|^r \le \sfrac12|k| \notag\\
   \Rightarrow\quad
   &|k-j|\ge|k|-|j|\ge\sfrac12|k| \notag\\
   \Rightarrow\quad
   &|j| \le |k|^r \le (2|k-j|)^r = 2^r|m|^r. \label{q:jkr}
\end{align}
We then bound $|k|^{-4}=|m+j|^{-4}$ from above and below subject to the constraints on $|j|$.
Noting that for $x\in(0,1)$, {by convexity} we can estimate
\begin{equation}
   (1+x)^{-4} 
   \ge 1-4x\;.
\end{equation}
This and \eqref{q:jkr} give us
\begin{align*}
   |m+j|^{-4} &\ge (|m|+|j|)^{-4} & &= |m|^{-4}(1+|j|/|m|)^{-4}\\
	&\ge |m|^{-4}(1-4|j|/|m|) & &\ge |m|^{-4}(1-2^{2+r}|m|^{r-1}).
\end{align*}
For the upper bound, we use the fact (readily seen by {convexity}), that for $x\in(0,\sfrac12]$
\begin{equation}\label{q:binombd}
   (1-x)^{-4} \le 1 + {30}x.
\end{equation}
Analogous reasoning then gives us
\begin{equation}\label{q:aux31}\begin{aligned}
   |m+j|^{-4} &\le (|m|-|j|)^{-4} & &= |m|^{-4}(1-|j|/|m|)^{-4}\\
	&\le |m|^{-4}(1+ {30}|j|/|m|) & &\le |m|^{-4}(1 + {30\cdot2^r}|m|^{r-1}).
\end{aligned}\end{equation}
From $\kappa\ge4\kpg^2$, we have $|j|<\kappa^r\le\kappa/4\le\sfrac14|k|<\sfrac13|k|$, so $\sfrac32|j|\le\sfrac12|k|$ and $|j|\le\sfrac12(|k|-|j|)\le\sfrac12|k-j|=\sfrac12|m|$, so we can use \eqref{q:binombd} in \eqref{q:aux31}.
We have thus shown that
\begin{equation}\label{q:bdmpj}
   |m|^{-4} - 8\,|m|^{r-5} \le |k|^{-4} \le |m|^{-4} + {60}\,|m|^{r-5}.
\end{equation}
This can be improved slightly by taking $r=(\beta-1)/(\alpha+1)$ instead of $\sfrac12$ and adjusting the constants.
We note that with finite-mode sources, there is no need to split $S_k$ and $|m+j|^{-4}$ is bounded by $|m|^{-4}\pm4\kpg|m|^{-5}$.

Instead of computing individual $S_k^\ll$, we proceed directly to the dyadic sum
\begin{equation}\label{q:aux2}
   \sum_{\kappa\le|k|<2\kappa} |k|^{-4} S_k^{\ll}
   = \sum_{1\le|j|<\kappa^r} |\gamma_j|^2 \sum_{\kappa\le|k|<2\kappa} |k|^{-4}|k-j|^{2\beta}(U_e^2\xi_{kj}^2+U_f^2\upsilon_{kj}^2).
\end{equation}
Defining spherical coordinates $(\rho,\vfi,\phi)$ w.r.t.\ $m$, i.e.
\begin{equation*}
   m = \rho (\sin\vfi\cos\phi,\sin\vfi\sin\phi,\cos\vfi)
	=: \rho\hat m,
\end{equation*}
we compute
\begin{equation}
  e_m\cdot j = j_x\sin\phi - j_y\cos\phi
  \text{ and }
  f_m\cdot j = j_x\cos\vfi\cos\phi + j_y\cos\vfi\sin\phi - j_z\sin\vfi.
\end{equation}
We approximate the $k$-sum by an integral over $m$, and in view of \eqref{q:aux31}, replace $|k|^{-4}$ by $|m|^{-4}$, so that
\intomargin
\begin{equation}\label{q:i1}\begin{aligned}
    & \tssum_{\kappa\le|k|<2\kappa}\,|k|^{-4}|k-j|^{2\beta}[U_e^2\xi_{kj}^2 + U_f^2\ups_{kj}^2] = \rmd_1(\kappa) + {}\\
    &\int_0^{2\pi}\!\!\int_0^\pi\!\!\int_{r_j(\kappa,\vfi,\phi)}^{r_j(2\kappa,\vfi,\phi)} [U_f^2(j_x\cos\vfi\cos\phi + j_y\cos\vfi\sin\phi - j_z\sin\vfi)^2\\
	&\hbox to110pt{}+ U_e^2(j_x\sin\phi-j_y\cos\phi)^2] \;\rho^{2\beta-2} \drho\sin\vfi\dvfi\dphi\\
    &=: \rmd_1(\kappa) + I_1(\kappa)
\end{aligned}\end{equation}
where $\rmd_1$ is a remainder to be bounded below, and where the radial limit $r_j(\lambda,\vfi,\phi)$, with $\lambda\in\{\kappa,2\kappa\}$, is determined by solving $|m+j|^2=\lambda^2$ for $\rho=|m|$,
\begin{equation}\label{q:rj}\begin{aligned}
  r_j(\lambda,\vfi,\phi) &= -j\cdot\hat m + \sqrt{\lambda^2-|j_\perp|^2}\qquad & &\text{with } |j_\perp|^2=|j|^2-(j\cdot\hat m)^2,\\
	&= \lambda - j\cdot\hat m - |j_\perp|^2/(2\lambda) + \cdots & &\text{for }\lambda\gg|j|.
\end{aligned}\end{equation}
Finally, we modify the region of integration, replacing the $\rho$-limit $r_j(\lambda,\vfi,\phi)$ by $\lambda$,
\intomargin
\begin{equation}\label{q:i1a}\begin{aligned}
    &\int_0^{2\pi}\!\!\int_0^\pi\!\!\int_{\kappa}^{2\kappa} [U_f^2(j_x\cos\vfi\cos\phi + j_y\cos\vfi\sin\phi - j_z\sin\vfi)^2\\
	&\hbox to110pt{}+ U_e^2(j_x\sin\phi-j_y\cos\phi)^2] \;\rho^{2\beta-2} \drho\sin\vfi\dvfi\dphi\\
    &\qquad= \kappa^{2\beta-1}i_2(2\beta-1)\Bigl[2\pi|j_h|^2 U_e^2 + \Bigl(\frac{2\pi}3|j|^2 + 2\pi j_z^2\Bigr)U_f^2\Bigr] =: I_2(\kappa)
\end{aligned}\end{equation}
where $i_2(s):=(2^s-1)/s$.
We write $\rmd_2(\kappa):=I_1(\kappa)-I_2(\kappa)$.

Our approximation for $\EV\|\Pkk\vtt\|_{L^2}^2$ is obtained by using $I_2(\kappa)$ in \eqref{q:aux2},
\begin{equation}
  \sum_{|j|<\kappa^r}|\gamma_j|^2 I_2 =
  \kappa^{2\beta-1}i_2(2\beta-1) \sum_{|j|<\kappa^r}|\gamma_j|^2 \Bigl[2\pi|j_h|^2 U_e^2 + \Bigl(\frac{2\pi}3|j|^2 + 2\pi j_z^2\Bigr)U_f^2\Bigr].
\end{equation}
In the isotropic case, $U_e=U_f\equiv U$, this reduces to
\intomargin
\begin{equation}\label{q:sll}
  \sum_{\kappa\le|k|<2\kappa} \!\!\!|k|^{-4}S_k^\ll
  = \kappa^{2\beta-1}\frac{U^2}{3\pi^2}i_2(2\beta-1)\|\Proj_{1,\kappa^r}\gb^{-1}g\|_{L^2}^2 + \sum_{1\le|j|<\kappa^r}|\gamma_j|^2(\rmd_1+\rmd_2)\end{equation}

We now bound the remainders $\rmd_1$ and $\rmd_2$.
The remainder $\rmd_2$ was incurred by replacing $r_j(\lambda,\vfi,\phi)$ in \eqref{q:i1} by $\lambda$ in \eqref{q:i1a}.
Now from \eqref{q:rj}
{since $\lambda \gg |j|$, we can bound
\begin{equation}
   |\lambda-r_j| \le \lambda+j\cdot \hat m - \sqrt{\lambda^2 +(j\cdot \hat m)^2 -|j|^2} \le \lambda + |j|-\sqrt{\lambda^2 -|j|^2}  \le 2|j|,
\end{equation}}
so we can bound $\rmd_2$ by integrating (a bound on the integrand) over two spherical shells of thickness $4|j|\le4\kappa^r$ at $\lambda=\kappa$ and $2\kappa$.
Bounding the integrand by $U^2|j|^2\rho^{2\beta-2}$, which is largest (since $\beta-1<0$) for smallest $\rho$, we have
\begin{align*}
   |\rmd_2(\kappa)| &\le \omega_d \sum\nolimits_{\lambda\in\{\kappa,2\kappa\}}\int_{\lambda-2\kappa^r}^{\lambda+2\kappa^r} U^2|j|^2 \rho^{2\beta-2} \drho  \\
   & \le \omega_d U^2|j|^2 4\kappa^{r}\left[(\kappa-2\kappa^r)^{2\beta-2}+(2\kappa-2\kappa^r)^{2\beta-2} \right] \\
	&\le \omega_d U^2|j|^2 8 \kappa^{r}(\kappa-2\kappa^r)^{2\beta-2} \\
	&\le  2^{5-2\beta}\omega_d U^2|j|^2 \kappa^{2\beta-2+r}.
\end{align*}
Therefore, bounding $\|\Proj_{1,\kappa^r}\gb^{-1}g\|_{L^2}^2\le\|\gb^{-1}g\|_{L^2}^2$,
\begin{equation}\label{q:E2}
   \tssum_{1\le|j|<\kappa^r}\,|\gamma_j|^2|\rmd_2(\kappa)|
	\le {\frac{2^{2-2\beta}}{\pi^3}}\omega_d U^2\kappa^{2\beta-2+r}\|\gb^{-1}g\|_{L^2}^2.
\end{equation}
Next, the remainder $\rmd_1$ incurred in \eqref{q:i1} is bounded by replacing $|m|^{-4}$ there by ${60}|m|^{r-5}$,
giving [cf.\ the first term in \eqref{q:sll}]
\begin{equation}\label{q:E1}
   \tssum_{|j|<\kappa^r}'\,|\gamma_j|^2\rmd_1(\kappa)
	\le c(\beta,d,r)\kappa^{2\beta+r-2} U^2\|\gb^{-1}g\|_{L^2}^2.
\end{equation}
Together \eqref{q:E2}--\eqref{q:E1} give the second term in \eqref{q:bdE}.
We note that this took more work than in two dimensions, where the simpler ``geometric term'' $k\wedge j=k_xj_y-k_yj_x$ {in \cite{jolly-dw:bht}} allowed direct integration in $k$ rather than having to shift to $m=k-j$.

\subsection{Upper Bound for $\Var\|\Pkk\vtt\|^2$}

For this, we take $U_e=U_f=U$.
To bound the variance, we first compute
\begin{equation}\label{q:aux41}
   \EV\,{\|}\Pkk\vtt\|_{L^2}^4 = \EV\,\tssum_{kl}\,|\vtt_k|^2|\vtt_l|^2
	=: U^4\tssum_{kl}\,|k|^{-4}|l|^{-4}\EV\,|\vfi_k|^2|\vfi_l|^2
\end{equation}
where, here and in the rest of this subsection, $\tssum_{kl}$ is taken over $|k|$, $|l|\in[\kappa,2\kappa)$.
Assuming $g$ is deterministic, we have
\begin{equation}\label{q:aux40}\begin{aligned}
   \EV\,|\vfi_k|^2|\vfi_l|^2 = &\tssum_{ijmn}'\,|k-i|^\beta|l-j|^\beta|l-m|^\beta|k-n|^\beta\\
	&\bigl\{\,\xi_{ki}\xi_{lj}\xi_{lm}\xi_{kn}\,\EV\,V_{k-i} V_{l-j}\overline{V_{l-m} V_{k-n}}\,\gamma_i\gamma_j\overline{\gamma_m\gamma_n} + (*)'\\
	&+ \xi_{ki}\xi_{lj}\ups_{lm}\ups_{kn}\,\EV\,V_{k-i} V_{l-j}\overline{W_{l-m} W_{k-n}}\,\gamma_i\gamma_j\overline{\gamma_m\gamma_n} + (*)'\\
	&+ \xi_{ki}\ups_{lj}\xi_{lm}\ups_{kn}\,\EV\,V_{k-i} W_{l-j}\overline{V_{l-m} W_{k-n}}\,\gamma_i\gamma_j\overline{\gamma_m\gamma_n} + (*)'\\
	&+ \xi_{ki}\ups_{lj}\ups_{lm}\xi_{kn}\,\EV\,V_{k-i} W_{l-j}\overline{W_{l-m} V_{k-n}}\,\gamma_i\gamma_j\overline{\gamma_m\gamma_n} + (*)'\bigr\}
\end{aligned}\end{equation}
where $(*)'$ denotes the preceeding term with $\xi\leftrightarrow\ups$ and $V\leftrightarrow W$ swapped (but not their indices).

We start with the last term: here $\EV\,V_{k-i} W_{l-j}\overline{W_{l-m} V_{k-n}} = \EV\,V_{k-i}\overline{V_{k-n}}\,\EV\, W_{l-j}\overline{W_{k-n}}\ne0$ only when $k-i=k-n$ and $l-j=l-m \Iff n=i$ and $j=m$.  This last term then contributes
\begin{equation}
   S_{kl}^{(p)} = \tssum_{ij}'\,|k-i|^{2\beta}|l-j|^{2\beta} \xi_{ki}^2\ups_{lj}^2 |\gamma_i|^2|\gamma_j|^2.
\end{equation}
{To reduce clutter, we now write $q:=k-l$ and $r:=k+l$.}
In the penultimate term, $\EV\cdots = \EV V_{k-i}\overline{V_{l-m}}\,\EV W_{l-j}\overline{W_{k-n}}\ne0$ only if $k-i={l-m} \Iff m=i-q$ and $l-j=k-n \Iff n=j+q$, thus contributing
\begin{equation}
   S_{kl}^{(o)} = \tssum_{ij}'\,|k-i|^{2\beta}|l-j|^{2\beta} \xi_{ki}\xi_{l,i-q} \ups_{lj}\ups_{k,j+q}\, \gamma_i\gamma_j\overline{\gamma_{i-q}\gamma_{j+q}}.
\end{equation}
Similarly, in the second term $\EV\cdots=\EV V_{k-i} V_{l-j}\,\EV \overline{W_{l-m} W_{k-n}}\ne0$ only if $j-l=k-i \Iff j=-i+r$ and $n-k=l-m \Iff n=-m+r$, contributing (upon relabelling $m\mapsto j$)
\begin{equation}
   S_{kl}^{(h)} = \tssum_{ij}'\,|k-i|^{2\beta}|l-j|^{2\beta} \xi_{ki}\xi_{l,-i+r} \ups_{lj}\ups_{k,-j+r}\,\gamma_i\overline{\gamma_{i-r}\gamma_j}\gamma_{j-r}.
\end{equation}

The first term is the most involved.
Denoting $-y\ne x\ne y$ by $x\ne_\pm y$, the factor $\EV V_{k-i}V_{l-j}\overline{V_{l-m}V_{k-n}}\ne0$ only in the following cases:
\begin{align}
  &k-i = k-n \ne_\pm l-j = l-m & &\Iff & &i=n \ne j+q=m+q, \tag{a}\\
  &k-i = l-m \ne_\pm l-j = k-n & &\Iff & &i=m+q \ne n=j+q, \tag{b}\\
  &k-i = j-l \ne_\pm k-n = m-l & &\Iff & &j=-i+r \ne m=-n+r, \tag{c}\\
  &k-i = k-n = l-j = l-m & &\Iff & &i=n = j+q=m+q, \tag{d}\\
  &k-i = j-l = k-n = m-l & &\Iff & &j=m = -i+r=-n+r, \tag{e}\\
  &k-i = l-m = j-l = n-k & &\>\Rightarrow & &m=i-q,\; j=-i+r,\; n=-i+2k. \tag{f}
\end{align}
In cases (a)--(c), the $\EV\cdots=1$, while in cases (d)--(f), the $\EV\cdots=\ppt$.
Imposing these conditions in \eqref{q:aux40}, the first term is $S_{kl}^{(a)} + S_{kl}^{(b)} + S_{kl}^{(c)} + \ppt S_{kl}^{(d)} + \ppt S_{kl}^{(e)} + \ppt S_{kl}^{(f)}$, where (in all these sums, $k$, $l$, $q=k-l$ and $r=k+l$ are fixed)
\begin{align}
  S^{(a)}_{kl} &= \tssum_{|k-i|\ne|l-j|}'\,|k-i|^{2\beta}|l-j|^{2\beta}\xi_{ki}^2\xi_{lj}^2|\gamma_i|^2|\gamma_j|^2\\
  S^{(b)}_{kl} &= \tssum_{|k-i|\ne|l-j|}'\,|k-i|^{2\beta}|l-j|^{2\beta}\xi_{ki}\xi_{k,j+q}\xi_{lj}\xi_{l,i-q}\, \gamma_i\gamma_j\overline{\gamma_{i-q}\gamma_{j+q}}\\
  S^{(c)}_{kl} & = \tssum_{|k-i|\ne|l-m|}'\,|k-i|^{2\beta}|l-m|^{2\beta}\xi_{ki}\xi_{k,r-m}\xi_{lm}\xi_{l,r-i}\,\gamma_i\gamma_{r-i}\overline{\gamma_m\gamma_{r-m}}\notag\\
	&= \tssum_{|k-i|\ne|l-j|}'\,|k-i|^{2\beta}|l-j|^{2\beta} \xi_{ki}\xi_{l,r-i}\xi_{k,r-j}\xi_{lj}\,\gamma_i\overline{\gamma_{i-r}\gamma_j}\gamma_{j-r},\\
  S^{(d)}_{kl} &= \tssum_{i}'\,|k-i|^{4\beta}\xi_{ki}^2\xi_{l,i-q}^2\,|\gamma_i|^2|\gamma_{i-q}|^2\\
  S^{(e)}_{kl} &= \tssum_i'\,|k-i|^{4\beta}\xi_{ki}^2\xi_{l,{r-i }}^2\,|\gamma_i|^2|\gamma_{r-i}|^2\\
  S^{(f)}_{kl} &= \tssum_{i}'\,|k-i|^{4\beta}\xi_{ki}\xi_{l,r-i}\xi_{l,i-q}\xi_{k,2k-i}\,\gamma_i\gamma_{i-2k}\overline{\gamma_{i-q}\gamma_{i-r}}.
\end{align}
Analogously, the first $(*)'$ in \eqref{q:aux40} is $S_{kl}^{(a')} + \cdots + \ppt S_{kl}^{(f')}$ with $\ups$ replacing $\xi$.

Returning to the variance, we have
\begin{equation}\begin{aligned}
   \Var\|\Pkk\vtt\|_{L^2}^2 &= \EV\,\|\Pkk\vtt\|_{L^2}^4 - \bigl(\EV\,\|\Pkk\vtt\|_{L^2}^2\bigr)^2\\
	&= U^4 \tssum_{kl}\,|k|^{-4}|l|^{-4} \bigl( \EV\,|\vfi_k|^2|\vfi_l|^2 - \EV|\vfi_k|^2\EV|\vfi_l|^2\bigr).
\end{aligned}\end{equation}
Now
\begin{equation}
   \EV|\vfi_k|^2\,\EV|\vfi_l|^2 = S_{kl}^{(a)} + S_{kl}^{(d)} + S_{kl}^{(e)} + S_{kl}^{(a')} + S_{kl}^{(d')} + S_{kl}^{(e')} + 2 S_{kl}^{(p)}
\end{equation}
where the factor of $2$ on $S_{kl}^{(p)}$ came from its $(*)'$.
This gives us
\begin{equation}\begin{aligned}
    \Var\|\Pkk\vtt\|_{L^2}^2 = U^4 &\tssum_{kl}\,|k|^{-4}|l|^{-4} \bigl( S_{kl}^{(b)} + S_{kl}^{(c)} + S_{kl}^{(b')} + S_{kl}^{(c')}\\
	&+ (\ppt-1) \bigl[ S_{kl}^{(d)} + S_{kl}^{(e)} + S_{kl}^{(d')} + S_{kl}^{(e')} \bigr]
	+ \ppt \bigl[ S_{kl}^{(f)} + S_{kl}^{(f')} \bigr]\\
	&+ S_{kl}^{(h)} + S_{kl}^{(o)} + S_{kl}^{(h')} + S_{kl}^{(o')} \bigr)
\end{aligned}\end{equation}
where $S_{kl}^{(h')}$ is $S_{kl}^{(h)}$ with $\xi$ and $\ups$ (but not their indices) swapped, arising from the $(*)'$ of the $(h)$ term in \eqref{q:aux40}, and analogously for $S_{kl}^{(o')}$.
So far, no approximation has been made, nor has the finite-mode source assumption been used.

We now invoke the assumption that $\gamma_j=0$ whenever $|j|\ge\kpg$.
Since $|k|\gg\kpg$, only one of $\gamma_i$ and $\gamma_{i-2k}$ can be non-zero, so the factor $\gamma_i\gamma_{i-2k}$ in $S^{(f)}_{kl}$ vanishes, killing the term; obviously $S_{kl}^{(f')}=0$ as well.

Next, we treat the contribution of $S_{kl}^{(d)}$: due to the terms $|\gamma_i|^2|\gamma_{i-q}|^2$, we must have $|q|<2\kpg$ for the terms containing it to be non-zero.
Rewriting the $l$-sum over $q=k-l$ and using the fact that $|k|\gg\kpg$, we approximate $|k-q|\simeq|k|\simeq|k-i|$ and bound $|\xi_{ki}|=|e_{k-i}\cdot i|\le|i|$ to get
\begin{equation}\begin{aligned}
   \tssum_{kl}\,|k|^{-4}&|l|^{-4}\,S_{kl}^{(d)}
	= \tssum_{kq}\,|k|^{-4}|k-q|^{-4}\tssum_i'\,|k-i|^{4\beta}\xi_{ki}^2\xi_{k-q,i-q}^2\,|\gamma_i|^2|\gamma_{i-q}|^2\\
	&\lesssim \tssum_{k}\,|k|^{4\beta-8}\tssum_{iq}'\,|i|^2|i-q|^2|\gamma_i|^2|\gamma_{i-q}|^2\\
	&= \tssum_k\,|k|^{4\beta-8}\tssum_i'\,|i|^2|\gamma_i|^2\tssum_q'|i-q|^2|\gamma_{i-q}|^2\\
	&= \Bigl(\tssum_k\,|k|^{4\beta-8}\Bigr)\Bigl(\tssum_i'\,|i|^2|\gamma_i|^2\Bigr)^2\\
	&= {\frac{\omega_3}{(2\pi)^{6}}} i_2(4\beta+d-8)\kappa^{4\beta+d-8}\|\gb^{-1}g\|_{L^2}^4.
\end{aligned}\end{equation}
where for the penultimate equality we have changed the last $\sum_q$ to go over $n=i-q$ and re-labelled.
Since we can bound $|\ups_{ki}|=|f_{k-i}\cdot i|\le|i|$ as with $|\xi_{ki}|$, this bound also holds for the contribution of $S_{kl}^{(d')}$.
An analogous argument gives us the bound
\begin{equation}
   \tssum_{kl}\,|k|^{-4}|l|^{-4}\bigl( S_{kl}^{(e)} + S_{kl}^{(e')} \bigr)
	\lesssim {\frac{2\omega_3}{(2\pi)^{6}}} i_2(4\beta+d-8)\kappa^{4\beta+d-8}\|\gb^{-1}g\|_{L^2}^4.
\end{equation}

We bound the contribution of $S_{kl}^{(b)}$ as follows:
\begin{equation}\begin{aligned}
   \tssum_{kl}\,&|k|^{-4}|l|^{-4}\tssum_{ij}'\,|k-i|^{2\beta}|l-j|^{2\beta}\xi_{ki}\xi_{k,j+q}\xi_{lj}\xi_{l,i-q}\,\gamma_i\gamma_j\overline{\gamma_{i-q}\gamma_{j+q}}\\
	&= \tssum_{kq}\,|k|^{-4}|k-q|^{-4}\tssum_{ij}'\,|k-i|^{2\beta}|k-q-j|^{2\beta}\cdots\\
	&\lesssim \tssum_{k}\,|k|^{4\beta-8}\,\tssum_{ij}'\,|i|\,|\gamma_i|\,|j|\,|\gamma_j| \tssum_q'|i-q|\,|\gamma_{i-q}|\,|j+q|\,|\gamma_{j+q}|\\
	&\le \sfrac12\,\tssum_{k}\,|k|^{4\beta-8}\,\tssum_{ij}'\,|i|\,|\gamma_i|\,|j|\,|\gamma_j| \tssum_q'\bigl(|i-q|^2|\gamma_{i-q}|^2 + |j+q|^2|\gamma_{j+q}|^2\bigr)\\
	&= \tssum_k\,|k|^{4\beta-8}\,\tssum_i'\,|i|\,|\gamma_i|\,\tssum_j'\,|j|\,|\gamma_j|\,\tssum_n'\,|n|^2|\gamma_n|^2\\
	&= {\frac{\omega_3}{(2\pi)^{3}}} i_2(4\beta+d-8)\kappa^{4\beta+d-8}\,\ggfosq{\gb^{-1}g}\|\gb^{-1}g\|_{L^2}^2\,.
\end{aligned}\end{equation}
Obviously this bound also applied to the contribution of $S_{kl}^{(b')}$, and by inspection, also to those of $S_{kl}^{(o)}$ and $S_{kl}^{(o')}$.
Similarly, we bound
\begin{equation}\begin{aligned}
   \tssum_{kl}\,&|k|^{-4}|l|^{-4}\tssum_{ij}'\,|k-i|^{2\beta}|l-j|^{2\beta}\xi_{ki}\xi_{l,r-i}\xi_{lj}\xi_{k,r-j}\,\gamma_i\overline{\gamma_{i-r}\gamma_j}\gamma_{j-r}\\
	&= \tssum_{kr}\,|k|^{-4}|r-k|^{-4}\tssum_{ij}'\,|k-i|^{2\beta}|r-k-j|^{2\beta}\cdots\\
	&\lesssim \tssum_{k}\,|k|^{4\beta-8}\,\tssum_{ij}'\,|i|\,|\gamma_i|\,|j|\,|\gamma_j| \tssum_r'|i-r|\,|\gamma_{i-r}|\,|j-r|\,|\gamma_{j-r}|\\
	&\le \sfrac12\,\tssum_{k}\,|k|^{4\beta-8}\,\tssum_{ij}'\,|i|\,|\gamma_i|\,|j|\,|\gamma_j| \tssum_r'\bigl(|i-r|^2|\gamma_{i-r}|^2 + |j-r|^2|\gamma_{j-r}|^2\bigr)\\
	&= \tssum_k\,|k|^{4\beta-8}\,\tssum_i'\,|i|\,|\gamma_i|\,\tssum_j'\,|j|\,|\gamma_j|\,\tssum_n'\,|n|^2|\gamma_n|^2\\
	&= {\frac{\omega_3}{(2\pi)^{3}}} i_2(4\beta+d-8)\kappa^{4\beta+d-8}\,\ggfosq{\gb^{-1}g}\|\gb^{-1}g\|_{L^2}^2\,,
\end{aligned}\end{equation}
with the same bound applying for $S_{kl}^{(c')}$, $S_{kl}^{(h)}$ and $S_{kl}^{(h')}$.
Putting everything together gives
\intomargin
\begin{equation}
   \Var\|\Pkk\vtt\|_{L^2}^2 \lesssim U^4{\frac{\omega_3}{(2\pi)^3}} i_2(4\beta-5) \kappa^{4\beta-5} \|\gb^{-1}g\|_{L^2}^2 \bigl\{ 8 \ggfosq{\gb^{-1}g} + {\frac{4}{(2\pi)^3}}(\ppt-1) \|\gb^{-1}g\|_{L^2}^2 \bigr\},
\end{equation}
whence follows \eqref{q:ubV}.

\subsection{Bounding $\tht^{(\infty)}-\tht^{(1)}$}

Finally, we bound the remainder $\tht^{(\infty)}-\tht^{(1)}$ in each dyad ...
As before, we write $d=3$ and $\omega_d=4\pi$ to make it easier to adapt the proof to two dimensions.
We start by obtaining a bound for $|\vtt_k|$.
From \eqref{q:vttk}, we have
\begin{equation}
   |\vtt_k| \le 2U\mxr|k|^{-2}\tssum_j'\,|k-j|^\beta|j|\,|\gamma_j|
	=: 2U\mxr|k|^{-2}S_k.
\end{equation}
When $|k|<2\kpg$, we have using \eqref{q:gdef} and bounding $|k-j|^\beta\le1$,
\begin{equation}
   S_k \le \tssum_j' |j|\,|\gamma_j| = \ggfo{\gb^{-1}g}.
\end{equation}
For the case $|k|\ge2\kpg$, we split the sum into four parts [cf.\ \eqref{q:sk0}]
\begin{equation}\label{q:sk}
   S_k = \tssum_{|j|<\kpg}' + \tssum_{\kpg\le|j|<|k|^r} + \tssum_{|k|^r\le|j|<2|k|} + \tssum_{2|k|\le|j|}
	=: S_k^g + S_k^\ll + S_k^\simeq + S_k^>.
\end{equation}
For $S_k^g$, since $|j|<\kpg$ and $|k|\ge2\kpg$ we have $|k-j|\ge|k|-|j|\ge\sfrac12|k|$ and thus $|k-j|^\beta\le2^{-\beta}|k|^\beta$; this gives
\begin{equation}
   S_k^g \le 2^{-\beta}|k|^\beta \tssum_{|j|<\kpg}'\,|j|\,|\gamma_j| \le 2^{-\beta}|k|^\beta\ggfo{\gb^{-1}g}.
\end{equation}
Similarly for $S_k^\ll$, since $|j|<|k|^r\le\sfrac12|k|$ (the latter holds since $2\le\kpg^{1-r}$), we again have $|k-j|\ge|k|-|j|\ge\sfrac12|k|$ and, since $\alpha+1+d<0$,
\begin{equation}
   S_k^\ll \le 2^{-\beta}|k|^\beta c_g\tssum_{\kpg\le|j|<|k|^r}\,|j|^{\alpha+1} \le 2^{-\beta}|k|^\beta c_g\omega_d\kpg^{\alpha+1+d}/|\alpha+1+d|.
\end{equation}
For $S_k^\simeq$, we use $|j|<2|k|$ to bound $|k-j|\le|k|+|j|<3|k|$ and change the sum over $j$ to (a larger) one over $m=k-j$; using $|j|\ge|k|^r$ to bound $|j|^{\alpha+1}\le|k|^{(\alpha+1)r}$, we then get (assuming $\beta+d\ne0$, a harmless special case)
\begin{equation}\label{q:j2k}
   S_k^\simeq \le c_g|k|^{(\alpha+1)r}\sum_{1\le|m|<3|k|}\,|m|^\beta
	\le c_g\omega_d|k|^{(\alpha+1)r}\frac{(3|k|)^{\beta+d}-1}{\beta+d}.
\end{equation}
If $\beta+d>0$, the fraction is bounded by $(3|k|)^{\beta+d}/(\beta+d)$,
and for the rhs to be $\Oh(|k|^\beta)$, we need $(\alpha+1)r+d\le0$.
If $\beta+d<0$, the fraction is bounded by $1/|\beta+d|$,
and for the rhs to be $\Oh(|k|^\beta)$, we need $(\alpha+1)r\le\beta$.
Either way, we need $(\alpha+1)r\le\min\{\beta,-d\}$.
For $S_k^>$, we use $|j|\ge2|k|$ to bound $|j-k|\ge|j|-|k|\ge\sfrac12|j|$ and $|k-j|^\beta\le2^{-\beta}|j|^\beta$, to get
\begin{equation}
   S_k^> \le 2^{-\beta}c_g \tssum_{2|k|\le|j|}\,|j|^{\alpha+\beta+1}
	\le 2^{-\beta}c_g\omega_d\,(2|k|)^{\alpha+\beta+d+1}/|\alpha+\beta+d+1|
\end{equation}
since $\alpha+\beta+d+1<0$;
for the rhs to be $\Oh(|k|^\beta)$, we need $\alpha+d+1\le0$.
Putting these together, we can write
\begin{equation}\label{q:bdvttk}
   |\vtt_k| \le U\mxr c_1(g,\beta,d)|k|^{-2}K_\beta(|k|)
\end{equation}
where
\begin{equation}\label{q:kbeta}
   K_\beta(|k|) := \min\{1,(2\kpg)^{-\beta}|k|^\beta\}.
\end{equation}
Summing over $k$, this gives $\|\Pkk\vtt\|_{L^2}^2\le c\,\kappa^{2\beta-4+d}U^2\mxr^4\ggfosq{\gb^{-1}g}$
for large $\kappa$, i.e.\ the same $\kappa^{2\beta-1}$ dependence as $\EV\|\Pkk\vtt\|_{L^2}^2$ albeit with a worse constant and dependence on $\ggfosq{\gb^{-1}g}$ instead of $\|\gb^{-1}g\|_{L^2}^2$.

Writing $\delta\tht^{(n)}:=\tht^{(n)}-\tht^{(n-1)}$ with $\delta\tht^{(1)}=\vtt$, we have from \eqref{q:sfpi}
\begin{equation}\begin{aligned}
   &\delta\tht^{(n+1)} = \Delta^{-1}(u\cdot\gb\delta\tht^{(n)})\\
   \Rightarrow\quad
   &|\delta\tht^{(n+1)}_k| \le 2U\mxr\,|k|^{-2}\tssum_j'\,|k-j|^\beta|j|\,|\delta\tht^{(n)}_j|.
\end{aligned}\end{equation}
Therefore, since $\delta\tht^{(1)}=\vtt$ is already bounded in \eqref{q:bdvttk}, if we can show that
\begin{equation}\label{q:skbeta}
   S_k := \tssum_j'\,|k-j|^\beta|j|^{-1}K_\beta(|j|) \le \varpi\,K_\beta(|k|),
\end{equation}
we can, by choosing
\begin{equation}\label{q:umax}
   4U\mxr\varpi =: U/U_{max} =: \eps < 1,
\end{equation}
ensure the mode-wise convergence
\intomargin
\begin{equation}\label{q:tht18}\begin{aligned}
   |\tht^{(1)}_k-\tht^{(\infty)}_k| &= |\delta\tht^{(2)}_k+\delta\tht^{(3)}_k+\cdots|
	\le |\delta\tht^{(2)}_k| + |\delta\tht^{(3)}_k| + \cdots\\
	&\le \frac{c_1\eps}2|k|^{-2} K_\beta(|k|) + \frac{c_1\eps^2}4|k|^{-2} K_\beta(|k|) + \cdots
	\le c_1\eps\,|k|^{-2} K_\beta(|k|).
\end{aligned}\end{equation}

To show \eqref{q:skbeta}, we first consider $|k|<4\kpg$.
We split the sum as
\begin{equation}
   \tssum_j'\,|k-j|^\beta|j|^{-1}K_\beta(|j|)
	= \tssum_{|j|<8\kpg}' + \tssum_{8\kpg\le|j|}
	=: S_k^g + S_k^>.
\end{equation}
For $S_k^g$, we simply bound $|k-j|^\beta K_\beta(|j|)\le1$ to get
\begin{equation}
   S_k^g \le \tssum_{|j|<8\kpg}'\,|j|^{-1}
	\le \omega_d(8\kpg)^{d-1}/(d-1).
\end{equation}
As for $S_k^>$, $|k|<4\kpg$ and $|j|\ge8\kpg$ implies that $|j-k|\ge|j|-|k|\ge\sfrac12|j|$ and thus $|k-j|^\beta\le2^{-\beta}|j|^\beta$, leading to
\begin{equation}
   S_k^> \le 2^{-\beta}\tssum_{8\kpg\le|j|}\,|j|^{2\beta-1}
	\le 2^{-\beta}\omega_d(8\kpg)^{2\beta+d-1}/|2\beta+d-1|
\end{equation}
since $2\beta+d-1<0$.
Together $S_k^g$ and $S_k^>$ give
\begin{equation*}
   S_k\le 2\omega_d\max\{2^{-\beta-1}/|\beta+1|,\, 32\kpg^2\}\,K_\beta(|k|)\quad \text{for all }|k|<4\kpg.
\end{equation*}

For the case $|k|\ge4\kpg$, we split the sum four ways
\begin{equation}\begin{aligned}
   S_k &= \tssum_{|j|<2\kpg}' + \tssum_{2\kpg\le|j|<\frac12|k|} + \tssum_{\frac12|k|\le|j|<2|k|} + \tssum_{2|k|\le|j|}\\
	&=: S_k^1 + S_k^< + S_k^\simeq + S_k^>.
\end{aligned}\end{equation}
For $S_k^1$, since $|j|<2\kpg$ and $|k|\ge4\kpg$, we have $|k-j|\ge|k|-|j|\ge\sfrac12|k|$ and $|k-j|^\beta\le2^{-\beta}|k|^\beta$, giving
\begin{equation}
   S_k^1 \le 2^{-\beta}|k|^\beta\tssum_{|j|<2\kpg}'|j|^{-1}
	\le 2^{-\beta}|k|^\beta\omega_d (2\kpg)^{d-1}/|d-1|.
\end{equation}
Next, for $S_k^<$, since $|j|<\sfrac12|k|$ and thus $|k-j|\ge|k|-|j|\ge\sfrac12|k|$, we have
\begin{equation}
   S_k^< \le {{4\kappa_g}}^{-\beta}|k|^\beta\tssum_{2\kpg\le|j|}\,|j|^{\beta-1}
	\le 2^{-\beta}|k|^\beta\omega_d(2\kpg)^{\beta+d-1}/|\beta+d-1|
\end{equation}
assuming that $\beta+d-1<0$.
For $S_k^\simeq$, we change the summation variable from $j$ to $m=k-j$ as in \eqref{q:j2k} and use $|j|\ge\sfrac12|k|$ and thus $|j|^{\beta-1}\le2^{1-\beta}|k|^{\beta-1}$, to obtain (again assuming $\beta+d\ne0$)
\begin{equation}
   S_k^\simeq \le {\kappa_g^{-\beta}}2^{1-\beta}|k|^{\beta-1}\tssum_{|m|<3|k|}'\,|m|^\beta
	\le {\kappa_g^{-\beta}}2^{1-\beta}|k|^{\beta-1}\omega_d\frac{(3|k|)^{\beta+d}-1}{\beta+d}.
\end{equation}
If $\beta+d>0$, the rhs is $\Oh(|k|^\beta)$ if $\beta+d-1\le0$.
Finally, for $S_k^>$, we use $|j|\ge2|k|$ to bound $|j-k|\ge|j|-|k|\ge\sfrac12|j|$ and $|k-j|^\beta\le2^{-\beta}|j|^\beta$; this gives us
\begin{equation}
   S_k^> \le 2^{-\beta}\tssum_{2|k|\le|j|}\,|j|^{2\beta-1}
	\le 2^{-\beta}\omega_d(2|k|)^{2\beta+d-1}/|2\beta+d-1|.
\end{equation}
For the rhs to be $\Oh(|k|^\beta)$, we need $\beta+d-1\le0$.
Subject to the assumptions below, we have thus established \eqref{q:skbeta}.

Summing \eqref{q:tht18}, we have
\begin{equation}
   \|\Pkk\delta\tht\|_{L^2}^2 \le {8\pi^3 \varepsilon^2}c_1^2\,\tssum_{\kappa\le|k|<2\kappa}\,|k|^{-4}{K_\beta^2(|k|)}^2
	\le c_2(\cdots)|k|^{2\beta+d-4}
\end{equation}
with the second inequality valid since $|k|\ge2\kpg$.

Collecting our assumptions, we need
\begin{align}
   &4\le\kpg^{1-r}, & &\\
   &(\alpha+1)r \le \min\{\beta,-d\} & &(\text{obviating }\alpha+d+1<0),\\
   &\beta+d-1 < 0. & &
\end{align}
\end{proof}

\begin{proof}[Proof of Theorem~\ref{t:main}]
The proof of the theorem is similar to that of the 2d case \cite{jolly-dw:bht}, with some improvements (e.g., requiring less regularity on $u$).
As in the proof of Theorem~\ref{t:static}, we shall often write $d=3$ to help possible adaptation to $d=2$.

Thanks to the boundedness of $V_k$ and $W_k$, we have from \eqref{q:udef}
\begin{equation}\begin{aligned}
   \|u(\cdot,t)\|_{H^{1/2}}^2 &\le {8\pi^3} \tssum_k'\,|k|^{2\beta+1}(U_e^2|V_k(t)|^2 + U_f^2|W_k(t)|^2)\\
	&\le U^2\mxr^2\tssum_k'\,|k|^{2\beta+1}
	\le U^2\mxr^2/(2\beta+d+1)
\end{aligned}\end{equation}
for all $t$, assuming that $2\beta+d+1<0$.

Considering the iteration \eqref{q:itn} as a mapping $T:\tht^{(n)}\mapsto\tht^{(n+1)}$, convergence of the iterations \eqref{q:it0}--\eqref{q:itn} would follow from the contractivity of $T$.
To prove the latter, we write $\delta\tht^{(n)}:=\tht^{(n)}-\tht^{(n-1)}$ and observe that it satisfies
\begin{equation}\label{q:aux30}
   (\dy_t - \Delta)\delta\tht^{(n)} = -(u\cdot\gb)\delta\tht^{(n-1)}
   \quad\text{with}\quad \delta\tht^{(n)}(\cdot,0) = 0.
\end{equation}
Multiplying this by $\delta\tht^{(n)}$ in $L^2(\Dom)$, we find
\begin{equation}
   \frac12\ddt{\;}\|\delta\tht^{(n)}\|_{L^2}^2 + \|\gb\delta\tht^{(n)}\|_{L^2}^2
	= -({(u\cdot \gb)}\delta\tht^{(n-1)}),\delta\tht^{(n)})_{L^2}^{}.
\end{equation}
{We next bound the contribution from the advected term as}
\begin{equation}\begin{aligned}
   \bigl|(u\cdot\gb\delta\tht^{(n-1)},\delta\tht^{(n)})_{L^2}^{}\bigr|
	&\le \|u\|_{L^3}^{}\|\gb\delta\tht^{(n-1)}\|_{L^2}^{}\|\delta\tht^{(n)}\|_{L^6}^{}\\
	&\le c\,\|u\|_{H^{1/2}}^{}\|\gb\delta\tht^{(n-1)}\|_{L^2}^{}\|\gb\delta\tht^{(n)}\|_{L^2}^{}\\
	&\le \sfrac12\|\gb\delta\tht^{(n)}\|_{L^2}^2 + c\,\|u\|_{H^{1/2}}^2\|\gb\delta\tht^{(n-1)}\|_{L^2}^2\,.
\end{aligned}\end{equation}
Integrating \eqref{q:aux30} in time, we find
\begin{equation}\begin{aligned}
   \|\delta\tht^{(n)}(t)\|_{L^2}^2 &+ \int_0^t \|\gb\delta\tht^{(n)}(s)\|_{L^2}^2 \ds\\
	&\le c_1\,\|u\|_{L^\infty([0,t],H^{1/2}(\Dom))}^2 \int_0^t \|\delta\tht^{(n-1)}(s)\|_{L^2}^2 \ds\\
	&\le c_1\,\|u\|_{L^\infty([0,t],H^{1/2}(\Dom))}^2 \int_0^t \|\gb\delta\tht^{(n-1)}(s)\|_{L^2}^2 \ds,
\end{aligned}\end{equation}
so (pathwise) convergence of $\tht^{(n)}$ in $L^2([0,t],H^1(\Dom))$ would follow from
\begin{equation}\label{q:idr}
   c_1\,\|u\|_{L^\infty(0,\infty;H^{1/2}(\Dom))}^2 = {8\pi^3 c_1}U^2\mxr^2/(2\beta+d+1) < 1.
\end{equation}

We now turn our attention to  $\vtt$, given by
\begin{equation}
   \vtt(t) = \theta^{(1)}(t) + \Delta^{-1}g
	= \int_0^t \ex^{(t-s)\Delta}u(s)\cdot\gb\Delta^{-1}g \ds,
\end{equation}
and whose Fourier coefficients satisfy (taking the general {\em deterministic\/} $g$)
\begin{equation}\begin{aligned}
    \vtt_k(t) &= \int_0^t \ex^{(s-t)|k|^2}\tssum_{j}'\, |k-j|^\beta[\xi_{kj}V_{k-j}(s) + \upsilon_{kj}W_{k-j}(s)]\gamma_{j} \ds\\
    &= \tssum_{j}'\,|k-j|^\beta\gamma_{j} \int_0^t \ex^{(s-t)|k|^2}[U_e \xi_{kj}V_{k-j}(s) + U_f \upsilon_{kj}W_{k-j}(s)] \ds
\end{aligned}\end{equation}
where, as in the proof of Theorem~\ref{t:static}, $\xi_{kj}:=e_{k-j}\cdot j$ and $\upsilon_{kj}:=f_{k-j}\cdot j$.
Since $\EV V_j(s)\overline{W_k(r)}=0$, for clarity we put $U_f=0$ temporarily, restoring it in \eqref{q:aux35}.
We compute
\begin{equation}
  \EV\vtt_k(t)\overline{\vtt_k(t)}
	= U_e^2\, \tssum_{ij}'\,|k-i|^\beta|k-j|^\beta\gamma_j\overline{\gamma_i}\;\EV \int_0^t \{\cdots\}_{j}\ds \int_0^t \overline{\{\cdots\}}_{i}\dr.
\end{equation}
Now
\begin{equation}\begin{aligned}
   \EV \int_0^t \{\cdots\}_{j}\ds \int_0^t \overline{\{\cdots\}}_{i}\dr
	&= \int_0^t\!\!\int_0^t \ex^{(s+r-2t)|k|^2}\EV V_{k-j}(s)\overline{V_{k-i}(r)} \dr\ds\\
	&= \int_0^t\!\!\int_0^t \ex^{(s+r-2t)|k|^2} \Phi_{k-j}(s-r)\delta_{ij} \dr\ds.
\end{aligned}\end{equation}
As in the static case, the sum over $i,j$ then collapses to one over $j$:
\begin{equation}
   \EV|\vtt_k(t)|^2
	= U_e^2\,\tssum_{j}' |k-j|^{2\beta}\xi_{kj}^2 |\gamma_j|^2 \int_0^t\!\!\int_0^t \ex^{(s+r-2t)|k|^2} \Phi_{k-j}(s-r) \dr\ds.
\end{equation}
We note that, except for the time integrals, the sum is exactly that in \eqref{q:evttk}, having {\em not\/} assumed stochastic $g$.
Restoring $U_f$, we have
\intomargin
\begin{equation}\label{q:aux35}
   \EV|\vtt_k(t)|^2 = \tssum_{j}' |k-j|^{2\beta}(U_e^2\xi_{kj}^2+U_f^2\upsilon_{kj}^2) |\gamma_j|^2 \int_0^t\!\!\int_0^t \ex^{(s+r-2t)|k|^2} \Phi_{k-j}(s-r) \dr\ds.
\end{equation}

To handle the integrals, we recall the following result proved in \cite[(3.20)--(3.26)]{jolly-dw:bht}: with $\Phi$ as in \eqref{q:Phi},
\begin{equation}\label{q:intlem}\begin{aligned}
    |k|^4\lim_{t\to\infty}\int_0^t\int_0^t &\ex^{(s+r-2t)|k|^2}\Phi(\chi|s-r|) \ds\dr\\
    &= 1 + \frac{\chi}{|k|^2}\Phi'(0) + \cdots + \frac{\chi^{{n-1}}}{|k|^{2n}}\int_0^\infty \ex^{-s|k|^2/\chi}\Phi^{(n)}(s)\ds.
\end{aligned}\end{equation}
We note that there is a spurious factor of $\sfrac12$ in the last line of (3.24) in \cite{jolly-dw:bht}, but (3.25) which we used here is correct.

Using this in \eqref{q:aux35}, we thus have
\intomargin
\begin{equation}
  \lim_{t\to\infty}\EV|\vtt_k(t)|^2 = |k|^{-4}\tssum_j'\,|k-j|^{2\beta}|\gamma_j|^2(U_e^2\xi_{kj}^2 + U_f^2\upsilon_{kj}^2) \{1 + \chi_{k-j}^{}\Phi'(0)/|k|^2 + \cdots\}.
\end{equation}
The first term (the 1) of the bracket, being exactly \eqref{q:evttk}, recovers the static case.
The higher-order terms, as in \eqref{q:intlem}, give smaller remainders.
\end{proof}

\bigskip\hbox to\hsize{\qquad\hrulefill\qquad}\medskip

\nocite{vallis:aofd}

\bibliographystyle{siam}
\bibliography{all,dw}

\bigskip\hbox to\hsize{\qquad\hrulefill\qquad}\medskip

\end{document}